\documentclass[draft]{amsart} 

\usepackage{amsmath}
\usepackage{amssymb}
\usepackage{amsthm}
\usepackage[dvipdfmx]{graphicx}
\usepackage{bm}
\usepackage{esint}
\usepackage[inline]{enumitem}
\usepackage{color}

\theoremstyle{plain}
\newtheorem{theorem}{Theorem}[section]
\newtheorem{corollary}[theorem]{Corollary}
\newtheorem{lemma}[theorem]{Lemma}
\newtheorem{proposition}[theorem]{Proposition}

\theoremstyle{definition}
\newtheorem{definition}[theorem]{Definition}
\newtheorem{remark}[theorem]{Remark}

\theoremstyle{remark}

\numberwithin{theorem}{section}
\numberwithin{equation}{section}

\newcommand{\R}{\mathbb{R}}

\newcommand{\dist}{\mathrm{dist}}
\newcommand{\diam}{\mathrm{diam}}

\newcommand{\cl}{\overline}

\newcommand{\loc}{\mathrm{loc}}

\DeclareMathOperator*{\divergence}{div}

\newcommand{\capacity}{\mathrm{cap}}

\newcommand{\trinorm}[1]
{{
    \left\vert\kern-0.20ex\left\vert\kern-0.20ex\left\vert
    #1 
    \right\vert\kern-0.20ex\right\vert\kern-0.20ex\right\vert
}}

\begin{document}


\title[Global H\"{o}lder solvability]{Global H\"{o}lder solvability of second order elliptic equations with locally integrable lower-order coefficients}
\author{Takanobu Hara}
\email{takanobu.hara.math@gmail.com}
\address{Graduate School of Science, Tohoku University, 6-3, Aramaki Aza-Aoba, Aoba-ku, Sendai, Miyagi 980-8578, Japan}
\date{\today}


\begin{abstract}
We prove the existence of globally H\"{o}lder continuous solutions to certain elliptic partial differential equations with lower-order terms.
Our result is applicable to coefficients controlled by a negative power of the distance from the boundary
of the domain
and significantly improves Theorem 8.30 in Gilbarg and Trudinger (1983).
The proof is derived by applying the strategy of Ancona (1986) to a new Morrey-type space.
\end{abstract}



\maketitle


\section{Introduction}\label{sec:introduction}

Regularity theory for second-order elliptic equations is central to the study of partial differential equations.
In particular, the existence of globally continuous solutions to Dirichlet problems is essential in both theoretical and applied contexts.

In this paper, we are concerned with the global H\"{o}lder solvability of the Dirichlet problem
\begin{equation}\label{eqn:DE}
\begin{cases}
- \divergence(A \nabla u) + \bm{b} \cdot \nabla u + \mu u
=
\nu & \text{in} \ \Omega, \\
u = g & \text{on} \ \partial \Omega.
\end{cases}
\end{equation}
Here, $\Omega$ is a bounded domain in $\R^{n}$ ($n \ge 2$),
$A \in L^{\infty}(\Omega)^{n \times n}$ is a matrix valued function satisfying the uniform ellipticity condition
\begin{equation}\label{eqn:A} 
|\xi|^{2}
\le
A(x) \xi \cdot \xi \le L |\xi|^{2} \quad \forall \xi \in \R^{n}, \ \forall x \in \Omega
\end{equation}
with a fixed constant $1 \le L < \infty$.
The  $g$ is a H\"{o}lder continuous function on the boundary $\partial \Omega$ of $\Omega$.
We further assume that $\Omega$ satisfies the capacity density condition
\begin{equation}\label{eqn:CDC}
\exists \gamma > 0 \quad
\frac{ \capacity( \cl{B(\xi, R)} \setminus \Omega, B(\xi, 2R)) }{ \capacity( \cl{B(\xi, R)}, B(\xi, 2R)) } \ge \gamma \quad \forall R > 0, \ \forall \xi \in \partial \Omega,
\end{equation}
where  $B(x, r)$ is a ball centered at $x$ with radius $r > 0$,
and $\capacity(K, U)$ is the relative capacity of an open set $U \subset \R^{n}$ and a compact set $K \subset U$, which is defined by
\begin{equation*}\label{eqn:variational_capacity}
\capacity(K, U)
:=
\inf
\left\{
\int_{\R^{n}} |\nabla u|^{2} \, dx \colon u \in C_{c}^{\infty}(U), \ u \ge 1 \, \text{on} \, K
\right\}.
\end{equation*}
We temporarily assume that $\bm{b} \in L^{2}_{\loc}(\Omega)^{n}$ and $\mu, \nu \in \mathcal{M}(\Omega)$,
where $\mathcal{M}(\Omega)$ is the set of all locally finite signed measures on $\Omega$. 

In the main theorem (Theorem \ref{thm:main}),
we establish the existence of a globally H\"{o}lder continuous solution to \eqref{eqn:DE}
under suitable Morrey-type conditions on $\bm{b}$, $\mu$ and $\nu$. 
Our main tool for dealing with lower-order terms is the traditional Fredholm alternative, and the normed space to which it is applied is new. 
As a consequence, well-known results on boundary regularity are significantly improved.

\subsection{Background}

There are many prior works for existence and regularity results of solutions to \eqref{eqn:DE} with various aspects.
We address \eqref{eqn:DE} using its divergence structure and refer to \cite{MR0244627, MR1814364} for basics of weak solutions.
These classical monographs used $L^{p}$ spaces as conditions on the coefficients,
but particularly since 1980s, sharp conditions for interior regularity estimates have been widely studied.
Local H\"{o}lder estimates for equations with Morrey coefficients was studied in, e.g., 
\cite{MR0271383, MR0964029, MR998128, MR1283326}.
See also \cite{MR644024, MR853783, MR1354887, MR1461542} for further information.
There are not a few results on boundary regularity as well, which will be discussed later.

Let us recall known results for the global H\"{o}lder regularity of weak solutions to
\begin{equation}\label{eqn:DE0}
\begin{cases}
- \divergence(A \nabla u)
=
0 & \text{in} \ \Omega, \\
u = g & \text{on} \ \partial \Omega.
\end{cases}
\end{equation}
It is well known that \eqref{eqn:CDC}  is sufficient for the desired global estimate.
See, e.g., \cite{MR163053, MR1814364}.
For further information, see also \cite[p.130]{MR2305115}. 
The proof consists of three major steps.
(i) Prove an interior regularity estimate.
(ii) Prove a regularity estimate at each boundary point by using \eqref{eqn:CDC} and the result of (i).
(iii) If the result of (ii) holds at all boundary points, then the desired global regularity follows.
As a consequence, if \eqref{eqn:CDC} holds, then the operator
\begin{equation}\label{eqn:harmonic_extention}
C^{\beta_{0}}(\partial \Omega) \ni g \mapsto u \in C^{\beta_{0}}(\cl{\Omega})
\end{equation}
is bounded for some $\beta_{0} \in (0, 1)$.
Here, for $E \subset \R^{n}$ and $\beta \in (0, 1)$,
$C^{\beta}(E)$ is the H\"{o}lder space endowed with the norm
\[
\| u \|_{C^{\beta}(E)}
:=
\sup_{E} |u|
+
\diam(E)^{\beta}
\sup_{ \substack{x, y \in E \\ x \neq y} }
\frac{|u(x) - u(y)|}{|x - y|^{\beta}}.
\]
Conversely, \cite[Theorem 3]{MR1924196} (see also \cite[Lemma 3]{MR856511}) showed
that the boundedness of \eqref{eqn:harmonic_extention} yields \eqref{eqn:CDC} if $\Omega$ has no irregular point.

For \eqref{eqn:DE}, one approach is to directly repeat the above three steps.
This argument is well-recognized as standard in boundary regularity estimates (see, e.g., \cite[Theorem 8.30]{MR1814364}),
and is still considered basic (e.g., \cite{MR4491776, MR4492670, MR4662426}).
However, this approach has certain drawbacks in the step (ii).
Specifically, the fact that the condition \eqref{eqn:CDC} holds at all boundary points is not fully utilized. 
Additionally, since it requires extending the equation out of the domain, this strategy cannot be applied to locally integrable $\bm{b}$, $\mu$, and $\nu$.

An alternative approach to \eqref{eqn:DE} is to construct the Green function of \eqref{eqn:DE0} and regard the lower-order terms as a perturbation.
This approach is often used in the context of potential theory (e.g., \cite{MR0910581, MR1731467, MR1482989, MR1828387, MR2140204}).
The problems in (ii) above do not occur in this method because the Green function is a global concept.
As a result, under smoothness assumptions on $\partial \Omega$,
it is possible to deal with coefficients that diverge by negative powers of $\delta(x) := \dist(x, \partial \Omega)$, as in
\begin{equation}\label{eqn:example_b}
\delta(x)^{1 - \beta} |\bm{b}(x)| \in L^{\infty}(\Omega),
\end{equation}
\begin{equation}\label{eqn:example_mu}
\mu = c(x) m \quad \text{and} \quad \delta(x)^{2 - \beta} c(x) \in L^{\infty}(\Omega),
\end{equation}
where $\beta \in (0, 1)$ and $m$ is the Lebesgue measure.
A major drawback of this approach is that explicit estimates for Green functions are difficult to derive in domains with complex-shaped boundaries.
Notably, Ancona \cite{MR856511} overcame this obstacle by developing a framework that avoids relying on explicit formulas for the Green function altogether.
The key idea in this approach is to rely on global estimates, thereby avoiding the need for explicit expressions of the Green function.

However, the results in \cite{MR856511}  impose strong assumptions on interior regularity and do not provide any explicit conclusions.
In particular, there is still room for improvement, especially in relaxing the assumptions and clarifying the associated norm inequalities.

In the present paper, we use a global H\"{o}lder estimate in \cite{hara2023global} (see, Lemma \ref{lem:poisson} below).
By the argument in \cite[Theorem 3.2]{MR998128}, if $\nu \ge 0$, $u \in C^{\beta}(\cl{\Omega})$, $\bm{b} = 0$, and $\mu = 0$,
then, $\nu$ must satisfy a Morrey-type condition in Lemma \ref{lem:poisson}.
We develop the perturbation argument under conditions that arise naturally from the goal of obtaining H\"{o}lder continuous solutions.

\subsection{Result}

We control $\bm{b}$, $\mu$ and $\nu$ using the following Morrey-type norm and corresponding normed spaces.
The details of them will be discussed in Section \ref{sec:MS}.

\begin{definition}[{\cite{hara2023global}}]\label{def:floated_morrey}
For $0 \le \lambda \le n$, we define
\[
\mathsf{M}^{\lambda}(\Omega)
:=
\left\{
\nu \in \mathcal{M}(\Omega) \colon \trinorm{\nu}_{\lambda, \Omega} < \infty
\right\},
\]
where
\begin{equation}\label{en:def_norm}
\trinorm{\nu}_{\lambda, \Omega}
:=
\sup_{ \substack{x \in \Omega \\ 0 < r < \delta(x) / 2} } r^{- \lambda} |\nu|(B(x, r)).
\end{equation}
\end{definition}

Our main result is as follows.

\begin{theorem}\label{thm:main}
Assume \eqref{eqn:A} and \eqref{eqn:CDC}.
Suppose that
\begin{equation}\label{eqn:b}
|\bm{b}|^{2} m \in \mathsf{M}^{ n - 2 + 2 \beta }(\Omega),
\end{equation}
\begin{equation}\label{eqn:mu}
\mu \in \mathsf{M}^{ n - 2 + \beta }(\Omega),
\end{equation}
where $\beta \in (0, 1)$
and  that $\mu \ge 0$.
Then, for each $\nu \in \mathsf{M}^{ n - 2 + \beta }(\Omega)$ and $g \in C^{\beta}(\partial \Omega)$, there exists
a unique weak solution $u \in H^{1}_{\loc}(\Omega) \cap C(\cl{\Omega})$ to \eqref{eqn:DE}.
Moreover, there exists a positive constant $\beta_{\star}$ depending only on $n$, $L$, $\beta$ and $\gamma$ such that
\begin{equation}\label{eqn:main_esti}
\| u \|_{C^{\beta_{\star}}(\cl{\Omega})}
\le
C 
\left(
\diam(\Omega)^{\beta} \trinorm{ \nu }_{ n - 2 + \beta, \Omega}
+
\| g \|_{C^{\beta}(\partial \Omega)}
\right),
\end{equation}
where $C$ is a positive constant independent of $\nu$ and $g$.
\end{theorem}

\begin{remark}
The solution $u$ in Theorem \ref{thm:main} may not have finite energy.
Note that we do not assume that $\nu \in H^{-1}(\Omega)$ or that $g$ is the trace of an $H^{1}(\Omega)$ function.

There is no size restriction of the norm of $\bm{b}$ in Theorem \ref{thm:main}.
This existence result holds even if the problem is not coercive in the sense of bilinear form on $H_{0}^{1}(\Omega)$.
\end{remark}

The classes of $\bm{b}$, $\mu$ and $\nu$ in Theorem \ref{thm:main} are larger than in \cite[Theorems 8.29 and 8.30]{MR1814364}.
As already mentioned, \eqref{eqn:example_b} and \eqref{eqn:example_mu} provide examples in which differences between the two appear.

Theorem \ref{thm:main} is not a generalization of the Wiener criterion (e.g., \cite[Theorem 8.31]{MR1814364} and \cite[Section 4.2]{MR1461542}).
The uniform condition \eqref{eqn:CDC} is necessary for the norm estimate \eqref{eqn:main_esti},
but stronger than the divergence of the Wiener integral at each boundary point.
Theorem \ref{thm:main} utilizes this stronger assumption effectively.

We note here the limitations of Theorem \ref{thm:main}.
First, we have no information of the optimal value of $\beta_{\star}$.
Second, we do not know sharp conditions for existence of globally continuous solutions.
Finally, cancellation of the coefficients which has been used in recent studies on interior regularity estimates
(e.g., \cite{MR2760150, MR2852216, MR3334220, MR3455216}), are not utilized.
These are topics for future work.

\subsection*{Organization of the paper}
In Section \ref{sec:MS}, we discuss properties of $\mathsf{M}^{\lambda}(\Omega)$.
In Section \ref{sec:perturbation}, we show compactness of lower-order perturbations and
prove Theorem \ref{thm:main} for the homogeneous boundary data $g = 0$.
In Section \ref{sec:BVP}, we complete the proof of Theorem \ref{thm:main}.

\subsection*{Notation}
Throughout the paper, $\Omega \subsetneq \R^{n}$ is a bounded open set.
We denote by $\delta(x)$ the distance from the boundary $\Omega$.
\begin{itemize}
\item
$C_{c}(\Omega) :=$
the set of all continuous functions with compact support in $\Omega$.
\item
$C_{c}^{\infty}(\Omega) := C_{c}(\Omega) \cap C^{\infty}(\Omega)$.
\end{itemize}
We denote by $\mathcal{M}(\Omega)$ the set of all measures on $\Omega$ in the sense in \cite{MR2018901}.
Using the Riesz representation theorem, we identify them with continuous linear functional functionals on $C_{c}(\Omega)$.
When the Lebesgue measure must be indicate clearly, we use the letter $m$.
For a function $u$ on $B$, we use the notation $\fint_{B} u \,dx := m(B)^{-1} \int_{B} u \, dx$.
The letter $C$ denotes various constants.

\section{Morrey spaces and elliptic regularity}\label{sec:MS}

We first consider properties of the normed space $\mathsf{M}^{\lambda}(\Omega)$ in Definition \ref{def:floated_morrey}.

Let us recall some facts of traditional Morrey spaces.
Since the introduction by Morrey \cite{MR1501936},
Morrey spaces have been represented in various notations (e.g., \cite{MR1354887, MR1461542, MR3467116, MR4288177}).
In \cite[p.29]{MR3467116}, for a fixed exponent $0 \le \lambda \le n$,
the set of all signed Radon measures on $\R^{n}$ satisfying
\[
|\mu|(B(x, r)) \le C r^{\lambda} \quad \forall x \in \R^{n}, \ 0 < \forall r < \infty
\]
is denoted by $L^{1, \lambda}$, where $C$ is a constant independent of $x$ and $r$.
If $f$ belongs to the Lebesgue space $L^{q}(\R^{n})$, then $f m \in L^{1, n - n / q}$.
For a finite signed measure on a bounded domain, its zero extension to $\R^{n}$ is considered.

Our space $\mathsf{M}^{\lambda}(\Omega)$ differs from the above one in that the range of $r$ is restricted by $\delta(x) / 2$.
When this is not significant, the same arguments as in the aforementioned literature apply.
For instance, if $f \in L^{q}(\Omega)$, then, $f m \in \mathsf{M}^{n - n / q}(\Omega)$.
However, this restriction may lead to differences:
$c(x)$ in \eqref{eqn:example_mu} may not be integrable on $\Omega$,
but a simple calculation we can check that $c m \in \mathsf{M}^{n - 2 + \beta }(\Omega)$
(\cite[Proposition 6.1]{hara2023global}).

Since $\Omega$ is bounded, for any $0 \le \lambda_{1} \le \lambda_{2} \le n$,
we have 
\begin{equation}\label{eqn:esti_n/2}
\trinorm{\nu}_{ \lambda_{1}, \Omega} \le \diam(\Omega)^{ \lambda_{2} - \lambda_{1} } \trinorm{\nu}_{ \lambda_{2}, \Omega}.
\end{equation}

\begin{theorem}\label{thm:completeness}
The normed space $\left( \mathsf{M}^{\lambda}(\Omega), \trinorm{\cdot}_{\lambda, \Omega} \right)$ is a Banach space.
\end{theorem}

\begin{proof}
Let $\{ \mu_{j} \}$ be a Cauchy sequence in $\mathsf{M}^{ \lambda }(\Omega)$.
Then, for any $\epsilon > 0$, there exists $j_{\epsilon}$ such that
\[
|\mu_{j} - \mu_{i}|(B) \le \epsilon \, \diam(B)^{\lambda}
\]
whenever $j, i \ge j_{\epsilon}$ and $2B \subset \Omega$.
Then, we have
\[
\left| \int_{\Omega} \varphi \, d (\mu_{j} - \mu_{i}) \right|
\le
\epsilon \, \diam(B)^{\lambda}
\]
for all
\begin{equation}\label{eqn:cond_test-func}
\varphi \in C_{c}(B), \quad \| \varphi \|_{L^{\infty}(\Omega)} \le 1.
\end{equation}
If $K \subset \Omega$ is compact, then, we can choose finitely many balls $\{ B_{k} \}$ such that $2B_{k} \subset \Omega$ and $K \subset \bigcup_{k} B_{k}$.
Using \eqref{eqn:cond_test-func} and a partition of unity,
we find that $\{ \mu_{j} \}$ is bounded in the sense of the dual of $C_{c}(\Omega)$.
Therefore, there exists a subsequence $\{ \mu_{j_{k}} \}$ of $\{ \mu_{j} \}$ and $\mu \in \mathcal{M}(\Omega)$ such that $\mu_{j_{k}}$ converges to $\mu$ vaguely.

Fix a ball $B$ and $\varphi$ satisfying \eqref{eqn:cond_test-func} again.
Taking the limit $i \to \infty$ along the above subsequence, we obtain
\begin{equation}\label{eqn:completeness01}
\left| \int_{\Omega} \varphi \, d (\mu_{j} - \mu) \right|
\le
\epsilon \, \diam(B)^{\lambda}
\end{equation}
and
\[
\left| \int_{\Omega} \varphi \, d \mu \right|
\le
( \trinorm{ \mu_{j} }_{\lambda, \Omega} + \epsilon) \diam(B)^{\lambda}.
\]
It follows from assumption on $\varphi$ that
\[
| \mu |(B)
\le
( \trinorm{ \mu_{j} }_{\lambda, \Omega} + \epsilon) \diam(B)^{\lambda}.
\]
Therefore, $\mu \in \mathsf{M}^{ \lambda }(\Omega)$.
Using \eqref{eqn:completeness01} again, we obtain
\[
\trinorm{ \mu_{j} - \mu }_{\lambda, \Omega} \le \epsilon.
\]
Consequently, $\mu_{j} \to \mu$ in $\mathsf{M}^{ \lambda }(\Omega)$.
The uniqueness of $\mu$ and the convergence of the whole sequence follows from the usual manner.
\end{proof}

Next, we define weak solutions to \eqref{eqn:DE}.

\begin{definition}\label{def:weak-sol}
Let $\bm{b} \in L^{2}_{\loc}(\Omega)$, and let $\mu, \nu \in \mathcal{M}(\Omega)$.
We say that a function $u \in H^{1}_{\loc}(\Omega) \cap C(\Omega)$ is a weak solution to \eqref{eqn:DE} if
\[
\int_{\Omega} A \nabla u \cdot \nabla \varphi + \bm{b} \cdot \nabla u \varphi \, dx
+
\int_{\Omega} u \varphi \, d \mu
=
\int_{\Omega} \varphi \, d \nu
\]
for all $\varphi \in C_{c}^{\infty}(\Omega)$.
\end{definition}

Throughout the paper, we understand \eqref{eqn:DE} in the sense of Definition \ref{def:weak-sol}.

The following weak Harnack inequality can be found in e.g., \cite[Theorem 3.13]{MR1461542}. 

\begin{lemma}\label{lem:weak_harnack}
Suppose that \eqref{eqn:A}, \eqref{eqn:b} and \eqref{eqn:mu} hold for some $\beta \in (0, 1)$.
Let $u$ be a nonnegative weak supersolution to $- \divergence(A \nabla u) + \bm{b} \cdot \nabla u + \mu u = 0$ in $\Omega$.
Assume that $B(x, 2r) \subset \Omega$.
Then, we have
\[
\fint_{B(x, r)} u \, dx \le C \inf_{B(x, r)} u,
\]
where $C$ is a positive constant depending only on $n$, $L$, $\beta$,
$\trinorm{ |\bm{b}|^{2} m}_{ n - 2 + 2 \beta, \Omega}$ and $\trinorm{\mu}_{ n - 2 + \beta, \Omega}$.
\end{lemma}

\begin{proposition}\label{prop:positive_mu}
Suppose that \eqref{eqn:A}, \eqref{eqn:b} and \eqref{eqn:mu} hold for some $\beta \in (0, 1)$.
Assume further that $\mu \ge 0$.
Let $u \in H^{1}_{\loc}(\Omega) \cap C(\cl{\Omega})$ be a weak solution to
\begin{equation}\label{eqn:EV}
\begin{cases}
- \divergence (A \nabla u) + \bm{b} \cdot \nabla u + \mu u = 0 & \text{in} \ \Omega,
\\
u = 0 & \text{on} \ \partial \Omega.
\end{cases}
\end{equation}
Then, $u = 0$.
\end{proposition}

\begin{proof}
This follows from the strong maximum principle.
Let $M = \sup_{\Omega} u \ge 0$.
We note that
\[
- \divergence (A \nabla (M - u)) + \bm{b} \cdot \nabla (M - u) + \mu (M - u) = M \mu \ge 0 \quad \text{in} \ \Omega.
\]
Assume that $M > 0$, and consider the set $E := \{ x \in \Omega \colon u(x) = M \}$.
Take $x \in E$ such that $\delta(x) = \dist(E, \partial \Omega) > 0$.
By Lemma \ref{lem:weak_harnack}, we have
\[
\fint_{B(x, \delta(x / 2))} (M - u) \, dx \le C \inf_{B(x, \delta(x) / 2)} (M - u) = 0.
\]
Since $B(x, \delta(x) / 2) \subset E$,
it follows from an elementary geometrical consideration that $\dist(E, \partial \Omega) \le \delta(x) / 2$.
This contradicts to the definition of $x$. Therefore, $M = 0$.
By the same way, $\inf_{\Omega} u = 0$.
\end{proof}

For $\bm{b} = \bm{0}$, $\mu = 0$ and $g = 0$, the following existence theorem holds.

\begin{lemma}[{\cite{hara2023global}}]\label{lem:poisson}
Assume that \eqref{eqn:A} and \eqref{eqn:CDC} hold.
Suppose that $\nu \in \mathsf{M}^{ n - 2 + \beta }(\Omega)$ for some $\beta \in (0, 1)$.
Then, there exists a unique weak solution
$u \in H^{1}_{\loc}(\Omega) \cap C(\cl{\Omega})$ to
\begin{equation}\label{eqn:poisson}
\begin{cases}
- \divergence(A \nabla u)  = \nu & \text{in} \ \Omega, \\
u = 0 & \text{on} \ \partial \Omega.
\end{cases}
\end{equation}
Moreover, there exist positive constants $C_{1}$ and $\beta_{1}$ depending only on $n$, $L$, $\beta$ and $\gamma$ such that
\begin{equation}\label{eqn:hoelder_esti}
\| u \|_{C^{\beta_{1}}(\Omega)}
\le
C_{1} \, \diam(\Omega)^{ \beta } \trinorm{ \nu }_{ n - 2 + \beta, \Omega}.
\end{equation}
\end{lemma}

Finally, we introduce the following notation.

\begin{definition}
Suppose that $\beta \in (0, 1)$.
For $\nu \in \mathsf{M}^{ n - 2 + \beta }(\Omega)$,
we denote by $\mathbf{G}_{0} \nu$ the weak solution $u \in H^{1}_{\loc}(\Omega) \cap C(\cl{\Omega})$ to \eqref{eqn:poisson}.
\end{definition}

\section{Lower-order terms}\label{sec:perturbation}

Let us recall the Fredholm alternative.

\begin{lemma}[{\cite[Theorem 5.3]{MR1814364}}]\label{lem:fredholm}
Let $X$ be a normed space, and let $T$ be a compact linear operator from $X$ into itself.
Then, either (i) the homogeneous equation
\begin{equation*}
x - T x = 0
\end{equation*}
has a nontrivial solution $x \in X$, or (ii) for each $y \in X$, the equation
\[
x - T x = y
\]
has a unique solution $x \in X$.
Moreover, in case (ii), the operator $(I - T)^{-1}$ exists and is bounded.
\end{lemma}

We apply Lemma \ref{lem:fredholm} to the operator
\begin{equation}\label{eqn:T}
T \colon
\mathsf{M}^{ n - 2 + \beta }(\Omega)
\ni
\nu
\mapsto
T \nu
:=
- \left( \bm{b} \cdot \nabla + \mu \right) \mathbf{G}_{0} \nu
\in
\mathsf{M}^{ n - 2 + \beta }(\Omega).
\end{equation}

\begin{lemma}\label{lem:bound_of_lot}
Assume that \eqref{eqn:b} and \eqref{eqn:mu} hold. 
Then, the operator $T$ in \eqref{eqn:T} is a compact operator from $\mathsf{M}^{ n - 2 + \beta }(\Omega)$ into itself.
Moreover, we have
\begin{equation}\label{eqn;boundedness_of_t}
\begin{split}
& 
\trinorm{ T \nu }_{ n - 2 + \beta, \Omega}
\\
& \quad \le
C_{2} \, \diam(\Omega)^{\beta}
\left( \trinorm{ |\bm{b}|^{2} m }_{n - 2 + 2 \beta, \Omega}^{1 / 2} + \trinorm{\mu}_{n - 2 + \beta, \Omega} \right) \trinorm{\nu}_{n - 2 + \beta, \Omega}
\end{split}
\end{equation}
for all $\nu \in \mathsf{M}^{n - 2 + \beta}(\Omega)$.
\end{lemma}

\begin{proof}
Let $u = \mathbf{G}_{0} \nu$.
By \eqref{eqn:poisson}, we have
\begin{equation}\label{eqn:l-infty_esti}
\| u \|_{L^{\infty}(\Omega)}
\le
C_{1} \, \diam(\Omega)^{\beta} \trinorm{ \nu }_{ n - 2 + \beta, \Omega}.
\end{equation}
Let $B(x, r)$ be a ball such that $B(x, 4r) \subset \Omega$.
Take $\eta \in C_{c}^{\infty}(B(x, 2r))$ such that $\eta = 1$ on $B(x, r)$ and $|\nabla \eta| \le C / r$.
Testing \eqref{eqn:poisson} with $u \eta^{2}$, we obtain
\[
\begin{split}
\int_{B(x, r)} |\nabla u|^{2} \, dx
& \le
C \left(
\frac{1}{r^{2}} \int_{B(x, 2r)} |u|^{2} \, dx + \int_{B(x, 2r)} |u| \, d |\nu|
\right).
\end{split}
\]
By \eqref{eqn:esti_n/2}, we also get
\begin{equation}\label{eqn:cacioppolli}
\begin{split}
\int_{B(x, r)} |\nabla u|^{2} \, dx
\le
C \left(
\| u \|_{L^{\infty}(\Omega)}^{2} + \| u \|_{L^{\infty}(\Omega)} \trinorm{\nu}_{n - 2, \Omega}
\right) r^{n - 2}.
\end{split}
\end{equation}
The right-hand side is estimated by \eqref{eqn:l-infty_esti}.
Meanwhile, by H\"{o}lder's inequality, we have
\[
\int_{B(x, r)} \left| \bm{b} \cdot \nabla u \right| \, dx
\le
\left( \int_{B(x, r)} |\bm{b}|^{2} \, dx \right)^{1 / 2}
\left( \int_{B(x, r)} |\nabla u|^{2} \, dx \right)^{1 / 2}.
\]
Combining these inequalities with \eqref{eqn:b}, we obtain
\begin{equation}\label{eqn:hoelder_b}
\begin{split}
& 
\int_{B(x, r)} |\bm{b} \cdot \nabla u| \, dx 
\\
& \quad \le
C \, \diam(\Omega)^{\beta} \trinorm{ |\bm{b}|^{2} m }_{n - 2 + 2 \beta, \Omega}^{1 / 2} \trinorm{ \nu }_{n - 2 + \beta, \Omega} r^{n - 2 + \beta}.
\end{split}
\end{equation}
Meanwhile, by \eqref{eqn:mu} and \eqref{eqn:l-infty_esti}, we have
\[
\int_{B(x, r)} |u| \, d |\mu|
\le
C \diam(\Omega)^{\beta} \trinorm{ \mu }_{n - 2 + \beta, \Omega} \trinorm{ \nu }_{n - 2 + \beta, \Omega} r^{n - 2 + \beta}.
\]
By a simple covering argument, we find that \eqref{eqn;boundedness_of_t} holds.

Let us prove the compactness of $T$.
Let $\{ \nu_{j} \}$ be a bounded sequence of measures in $\mathsf{M}^{n - 2 + \beta}(\Omega)$, and assume that
$\trinorm{ \nu_{j} }_{n - 2 + \beta, \Omega} \le M < \infty$.
Set $u_{j} = \mathbf{G}_{0} \nu_{j}$.
Since $\{ u_{j} \}$ is bounded in $C^{\beta_{0}}(\cl{\Omega})$,
by the Ascoli-Arzel\`{a} theorem, we can take a subsequence of $\{ u_{j} \}$ and $u \in C(\cl{\Omega})$ such that
$u_{j} \to u$ uniformly in $\Omega$.
Meanwhile, by \eqref{eqn:cacioppolli}, we have
\[
\begin{split}
& 
\int_{B(x, r)} |\nabla (u_{j} - u_{i})|^{2} \, dx
\\
& \quad \le
C \left(
\| u_{j} - u_{i} \|_{L^{\infty}(\Omega)}^{2} + 2 \| u_{j} - u_{i} \|_{L^{\infty}(\Omega)} \diam(\Omega)^{\beta} M
\right) r^{n - 2}
\end{split}
\]
for all $i, j \ge 1$.
It follows from \eqref{eqn:b} that $\{ (\bm{b} \cdot \nabla u_{j}) m \}$ is a Cauchy sequence in $\mathsf{M}^{n - 2 + \beta}(\Omega)$.
Similarly, $\{ \mu u_{j} \}$ is a Cauchy sequence in $\mathsf{M}^{n - 2 + \beta}(\Omega)$.
By Theorem \ref{thm:completeness}, $T$ is compact.
\end{proof}

\begin{corollary}\label{cor:fredholm}
Assume that \eqref{eqn:b} and \eqref{eqn:mu} hold. 
Then, either (i) the homogeneous equation \eqref{eqn:EV} has 
a nontrivial solution $u \in H^{1}_{\loc}(\Omega) \cap C(\cl{\Omega})$, or (ii) for each $\nu \in \mathsf{M}^{ n - 2 + \beta }(\Omega)$, the equation
\begin{equation}\label{eqn:nonEV}
\begin{cases}
- \divergence (A \nabla u) + \bm{b} \cdot \nabla u + \mu u = \nu & \text{in} \ \Omega,
\\
u = 0 & \text{on} \ \partial \Omega.
\end{cases}
\end{equation}
has a unique solution $u \in H^{1}_{\loc}(\Omega) \cap C^{\beta_{1}}(\cl{\Omega})$.
Moreover, in case (ii), the operator
\begin{equation}\label{eqn:g_t}
\mathbf{G}_{T} \colon \mathsf{M}^{n - 2 + \beta}(\Omega) \ni \nu
\mapsto
\mathbf{G}_{T} \nu := u \in H^{1}_{\loc}(\Omega) \cap C^{\beta_{1}}(\cl{\Omega})
\end{equation}
exists and is bounded.
\end{corollary}

\begin{proof}
Assume that there is a non-trivial solution $\sigma \in \mathsf{M}^{ n - 2 + \beta }(\Omega)$ to
\begin{equation}\label{eqn:FH_hom}
\sigma - T \sigma = 0.
\end{equation}
Then, $u := \mathbf{G}_{0} \sigma \in H^{1}_{\loc}(\Omega) \cap C^{\beta_{1}}(\cl{\Omega})$ is a non-trivial solution to \eqref{eqn:EV}.
We prove the converse statement.
Assume the existence of a non-trivial solution $u \in H^{1}_{\loc}(\Omega) \cap C(\cl{\Omega})$ to \eqref{eqn:EV}.
Take a ball $B(x, 4r) \subset \Omega$ and $\eta \in C_{c}^{\infty}(B(x, 2r))$ such that $\eta = 1$ on $B(x, r)$ and $|\nabla \eta| \le C / r$.
Testing \eqref{eqn:EV} with $u \eta^{2}$, we get
\[
\begin{split}
\int_{B(x, 2r)} |\nabla u|^{2} \eta^{2} \, dx
& \le
\frac{C}{r^{2}} \int_{B(x, 2r)} u^{2} \, dx 
\\
& \quad + \left| \int_{B(x, 2r)} \bm{b} \cdot \nabla u u \eta^{2} \, dx + \int_{B(x, 2r)} u^{2} \eta^{2} \, d \mu \right|.
\end{split}
\]
By the Young inequality $ab \le (\epsilon / 2) a^{2} + (2 \epsilon)^{-1} b^{2}$ ($a, b, \epsilon \ge 0$), we have
\[
\left| \int_{B(x, 2r)} \bm{b} \cdot \nabla u u \eta^{2} \, dx \right|
\le
\frac{\epsilon}{2} \int_{B(x, 2r)} |\nabla u|^{2} \eta^{2} \, dx
+
\frac{1}{2 \epsilon} \int_{B(x, 2r)} |\bm{b}|^{2} u^{2} \eta^{2} \, dx.
\]
Combining these inequalities with \eqref{eqn:b}, \eqref{eqn:mu} and \eqref{eqn:esti_n/2}, we obtain
\[
\int_{B(x, r)} |\nabla u|^{2} \, dx \le C \| u \|_{L^{\infty}(\Omega)}^{2} r^{n - 2}.
\]
It follows from \eqref{eqn:b} that $(\bm{b} \cdot \nabla u) m \in \mathsf{M}^{ n - 2 + \beta }(\Omega)$.
Meanwhile, $\mu u \in \mathsf{M}^{ n - 2 + \beta }(\Omega)$ because $u$ is bounded and \eqref{eqn:mu} holds.
Therefore, $\sigma := - \divergence(A \nabla u)$ belongs to $\mathsf{M}^{ n - 2 + \beta }(\Omega)$.
It is also a non-trivial solution to \eqref{eqn:FH_hom}.

Assume that there is no non-trivial solution to \eqref{eqn:FH_hom}.
By Lemmas \ref{lem:fredholm} and \ref{lem:bound_of_lot}, for each $\nu \in \mathsf{M}^{ n - 2 + \beta }(\Omega)$, there exists a unique solution
$\sigma \in \mathsf{M}^{ n - 2 + \beta}(\Omega)$ to $\sigma - T \sigma = \nu$.
Then, $u := \mathbf{G}_{0} \sigma \in H^{1}_{\loc}(\Omega) \cap C^{\beta_{1}}(\cl{\Omega})$ satisfies
\[
- \divergence(A \nabla u) = \nu + T \sigma = \nu - (\bm{b} \cdot \nabla + \mu) u.
\]
Let us prove the uniqueness of $u$.
If there are two different solutions $u_{1}$ and $u_{2}$ to \eqref{eqn:nonEV}, then $v = u_{1} - u_{2}$ is a non-trivial solution to \eqref{eqn:EV}.
Since $\sigma := - \divergence(A \nabla v)$ belongs to $\mathsf{M}^{ n - 2 + \beta }(\Omega)$, this contradicts to assumption.
\end{proof}

\begin{remark}\label{rem:small}
Assume further that 
\[
\trinorm{ |\bm{b}|^{2} m }_{ n - 2 + 2 \beta, \Omega}^{1 / 2} + \trinorm{\mu}_{ n - 2 + \beta, \Omega} \le (2 C_{1} C_{2})^{-1}.
\]
where $C_{1}$ and $C_{2}$ are constants in Lemmas \ref{eqn:poisson} and \ref{lem:bound_of_lot}, respectively.
Then, we can get an explicit bound of \eqref{eqn:g_t} by the contractive mapping theorem.
\end{remark}

\section{Inhomogeneous boundary data}\label{sec:BVP}

Let us extend Corollary \ref{cor:fredholm} for $g (\neq 0) \in C^{\beta}( \partial \Omega )$.

\begin{lemma}\label{lem:BVP}
Let $g \in C^{\beta}(\partial \Omega)$.
Then, there exists a unique weak solution $w \in H^{1}_{\loc}(\Omega) \cap C(\cl{\Omega})$ to \eqref{eqn:DE0}.
Moreover, there exists positive constants $C$ and $0 < \beta_{0} \le \beta$ such that
\[
\| w \|_{C^{\beta_{0}}(\Omega)}
\le
C \| g \|_{C^{\beta}(\partial \Omega)}.
\]
Assume further that \eqref{eqn:b} and \eqref{eqn:mu} hold.
Then, $\bm{b} \cdot \nabla w + \mu w \in \mathsf{M}^{ n - 2 + \beta }(\Omega)$ and
\begin{equation}\label{eqn:esti_w}
\trinorm{ \bm{b} \cdot \nabla w + \mu w }_{ n - 2 + \beta, \Omega}
\le
C \| w \|_{L^{\infty}(\Omega)}.
\end{equation}
\end{lemma}

\begin{proof}
As mentioned in Section \ref{sec:introduction},
the existence of $w$ and its H\"{o}lder estimate are well-known (see e.g. \cite[Theorem 6.44]{MR2305115}).
By the comparison principle, we have
\[
\sup_{\Omega} w -  \inf_{\Omega} w \le \sup_{\partial \Omega} g - \inf_{\partial \Omega} g.
\]
As the proof of \eqref{eqn:cacioppolli}, we have
\[
\begin{split}
\int_{B(x, r)} |\nabla w|^{2} \, dx
& \le
C \left( \sup_{\Omega} w -  \inf_{\Omega} w \right)^{2} r^{n - 2}
\end{split}
\]
whenever $B(x, 4r) \subset \Omega$.
By \eqref{eqn:b} and $\eqref{eqn:mu}$, we obtain \eqref{eqn:esti_w}.
\end{proof}

\begin{theorem}\label{thm:inhom}
Suppose that \eqref{eqn:A}, \eqref{eqn:CDC}, \eqref{eqn:b} and \eqref{eqn:mu} hold.
Assume further that there is no non-trivial solution to \eqref{eqn:EV}.
Then, for each $\nu \in \mathsf{M}^{ n - 2 + \beta }(\Omega)$ and $g \in C^{\beta}(\partial \Omega)$, there exists
a unique weak solution $u \in H^{1}_{\loc}(\Omega) \cap C(\cl{\Omega})$ to \eqref{eqn:DE}.
Moreover, there exists a positive constant $\beta_{\star}$ depending only on $n$, $L$, $\beta$ and $\gamma$ satisfying \eqref{eqn:main_esti},
where $C$ is a positive constant independent of $\nu$ and $g$.
\end{theorem}

\begin{proof}
Let $w$ be a weak solution in Lemma \ref{lem:BVP}.
Consider the problem
\begin{equation}
\begin{cases}
- \divergence (A \nabla v) + \bm{b} \cdot \nabla v + \mu v
=
\nu - \bm{b} \cdot \nabla w - \mu w & \text{in} \ \Omega,
\\
v = 0 & \text{on} \ \partial \Omega.
\end{cases}
\end{equation}
The right-hand side is in $\mathsf{M}^{n - 2 + \beta}(\Omega)$.
By Corollary \ref{cor:fredholm},
this equation has a unique solution $v \in H^{1}_{\loc}(\Omega) \cap C^{\beta_{1}}(\cl{\Omega})$.
Then, $u = v + w \in C^{\beta_{\star}}(\cl{\Omega})$ is a solution satisfying \eqref{eqn:main_esti}, where $\beta_{\star} = \min\{ \beta_{1}, \beta_{0} \}$.
\end{proof}

We now completes the proof of Theorem \ref{thm:main}.

\begin{proof}[Proof of Theorem \ref{thm:main}]
By Proposition \ref{prop:positive_mu}, there is no non-trivial solution to \eqref{eqn:EV}.
By Theorem \ref{thm:inhom}, there exists a weak solution to \eqref{eqn:DE}.
The inequality \eqref{eqn:main_esti} follows from
Lemmas \ref{lem:poisson} and  \ref{lem:BVP} and the boundedness of $(I - T)^{-1}$ in Corollary \ref{cor:fredholm}.
\end{proof}

\section*{Acknowledgments}
The author would like to thank Prof. Hiroaki Aikawa for his valuable suggestions to improve the presentation of the paper.
The author also appreciates the helpful comments of the anonymous referee and the editor. 
This work was supported by JST CREST (doi:10.13039/501100003382) Grant Number JPMJCR18K3 and
JSPS KAKENHI (doi:10.13039/501100001691) Grant Number 23H03798.
The author also acknowledges the Research Institute for Mathematical Sciences, an International Joint Usage/Research Center located in Kyoto University.
This version of the manuscript has been proofread by ChatGPT to improve the language quality.



\bibliographystyle{abbrv} 
\bibliography{reference}


\end{document}